\documentclass[11pt]{amsart}

\usepackage{amsmath, amsfonts, amssymb,amsthm}
\usepackage{amstext}
\usepackage{mathrsfs}

\usepackage{float,epsf,subfigure}
\usepackage[all,cmtip]{xy}
\usepackage{hyperref}
\usepackage{algorithm}
\usepackage{algorithmic}
\usepackage{mathtools}
\usepackage{float}
\usepackage{bm}
\theoremstyle{plain}
\newtheorem{theorem}{Theorem}[section]

\newtheorem{proposition}[theorem]{Proposition}

\newtheorem{cor}[theorem]{Corollary}
\newtheorem{def-thm}[theorem]{Definition-Theorem}
\newtheorem{lemma}[theorem]{Lemma}

\newtheorem{re}[theorem]{Remark}
\newtheorem{defi}[theorem]{Definition}
\newtheorem{ex}[theorem]{Example}

\newtheorem*{thai}{Theorem I}
\newtheorem*{thaii}{Theorem II}
\newtheorem*{thaiii}{Theorem III}

\newtheorem*{cori}{Corollary I}
\newtheorem*{corii}{Corollary II}

\theoremstyle{definition}

\newtheorem{remark}[theorem]{Remark}

\def\min{\mathop{\mathrm{min}}}

\def\PP{\mathbb P}

{

\begin{document}
\title[Nevanlinna theory via holomorphic forms]{Nevanlinna theory via holomorphic forms}
\author[X.J. Dong \& S.S. Yang] {Xian-jing Dong \& Shuang-shuang Yang}

\address{School of Mathematical Sciences \\ Qufu Normal University \\ Qufu, 273165, P. R. China}
\email{xjdong05@126.com}

\address{Department of Mathematics \\ Nanjing University \\ Nanjing, 210093, P. R. China}
\email{ssyang1997@126.com}


\date{}
\maketitle \thispagestyle{empty} \setcounter{page}{1}

\begin{abstract}  This paper re-develops   Nevanlinna theory for meromorphic functions on $\mathbb C$ in the viewpoint of holomorphic forms.  
According to our observation,  Nevanlinna's functions can be formulated by a holomorphic form. Applying this thought to  Riemann 
surfaces,  one  then extends the definition of  Nevanlinna's functions using a holomorphic form $\mathscr S$.  With the new settings,  an analogue of Nevanlinna theory for the \emph{$\mathscr S$-exhausted Riemann surfaces} is obtained,  which  is viewed as  a generalization of the classical Nevanlinna theory for $\mathbb C$ and $\mathbb D.$

\end{abstract}

\vskip\baselineskip

\setlength\arraycolsep{2pt}

\section{Introduction}

 \subsection{Motivation}~
 
  The paper  is motivated  by the early work due to  the first named author \cite{Dong} (see also Atsuji \cite{atsuji}), who considered  Nevanlinna theory for  complete K\"ahler manifolds with non-positive sectional curvature.  To make it simpler,  instead of  K\"ahler manifolds, we restrict ourselves to the Riemann surfaces.  In what follows, we introduce it without going into the details.

Let $S$ be a non-compact Riemann surface equipped with a complete Hermitian metric $h$ of non-positive Gauss curvature $K.$ Fix a  reference point $o\in S.$ Set 
$$  \kappa(r)=\min\big\{K(x): x\in \overline{D(r)}\big\},$$
where $D(r)$ is a geodesic ball of radius $r$ centered at $o.$  Let $f$ be a nonconstant  meromorphic function on $S$. 
One can well define  the Nevanlinna's functions  $T_f(r), m_f(r,a)$ and  $N_f(r,a)$ (see Section 3.1 in \cite{Dong}).  The first named author showed that (see Theorem 1.2 in \cite{Dong}) for any $\delta>0$
$$(q-2)T_f(r)+T(r, \mathscr R)\leq \sum_{j=1}^q\overline{N}_f(r,a_j)+O\Big(\log T_f(r)+\log C(o,r,\delta)\Big)$$
 holds for all $r\in(0,\infty)$ outside a set $E_\delta$ of finite Lebesgue measure,  
where $C(o,r,\delta)$ is a positive function with the estimate (see (19) in \cite{Dong})
$$\log C(r,o,\delta)\leq O\big(r\sqrt{-\kappa(r)}+\delta\log r\big),$$
and   the curvature term $T(r, \mathscr R)$ is   bounded by
$$r^2\kappa(r)\leq T(r,\mathscr R)\leq0.$$
As showed as above,    $\log C(o,r,\delta)$ and $T(r, \mathscr R)$  are estimated, however, these two estimates are  rough and hard to improve by using the previous methods in \cite{Dong}. 
For example,  we consider $S=\mathbb D$, where $\mathbb D$ is the unit disc. According to  the conditions for  metrics, one needs to equip $\mathbb D$ with  Poincar\'e metric (of curvature $-1$).  In this situation, we have $\kappa(r)\equiv-1.$ This gives that 
$$-r^2\leq T(r,\mathscr R)\leq0.$$
However,    the best  lower bound  of $T(r,\mathscr R)$  is  $O(-r).$  

The main drawback  of the method in \cite{Dong} 
 is that the selectivity of  metrics is restricted, i.e., the metrics have to  be complete and non-positively curved, but this will  cause a rough estimate. 

\subsection{A viewpoint of  holomorphic forms}~

Let  $f$ be a nonconstant meromorphic function on  $\mathbb C,$ namely, the complex plane with standard Euclidean metric. Nevanlinna's  \emph{characteristic function} $T(r,f)$ of $f$ is well known \cite{Nev} as follows
$$T(r,f):=m(r,f)+N(r,f),$$
where 
$$m(r,f)=\int_{|z|=r}\log^+|f|\frac{d\theta}{2\pi}$$
and $$N(r, f)=\int_{0}^r\frac{n(t,f)-n(0,f)}{t}dt+n(0,f)\log r$$
are called the \emph{proximity function} and \emph{counting function} of $f$ respectively, in which, $n(r, f)$ denotes the number of poles of $f$ on the disc $D(r):=\{|z|< r\}.$
Characteristic function  is an important notion, it  characterizes the   growth of  meromorphic functions, using which Nevanlinna \cite{Nev}  in 1925 established two fundamental theorems, i.e.,  \emph{First Main Theorem} and \emph{Second Main Theorem},    named  Nevanlinna theory.  
Nevanlinna theory plays a central  role in complex analysis and   hyperbolic geometry. 
Roughly speaking, this theory 
 studies the size  of  images of meromorphic functions  or mappings in   target spaces, it  
is a great generalization of  the Little Picard Theorem saying that a  meromorphic function must be a constant if it  omits   three distinct values in $\overline{\mathbb C}.$ There are plenty of excellent results in Nevanlinna theory, for example,  the reader may refer to \cite{gri, gri1, Lang, ru1, Shiff, stoll, wong, wu} and see also \cite{atsuji, chuang, Dong, he-ru, Hu, Sibony-Paun, Sakai, shiffman, Ya}.

In this paper,  we shall investigate  Nevanlinna theory from the viewpoint of  holomorphic forms $\mathscr S.$   With this idea,  one generalizes  the Nevanlinna theory  to one class of  non-compact Riemann surfaces  which we call the \emph{$\mathscr S$-exhausted  Riemann surfaces}.
To begin with, let us   describe  how   Nevanlinna's  functions  can be formulated by  a holomorphic form $\mathscr S.$ 
To see it more clearly, we shall  use  the Ahlfords' characteristic function 
\begin{equation}\label{ttt}
T_f(r)=\int_1^r\frac{dt}{t}\int_{D(t)}f^*\omega_{FS},
\end{equation}
which is equivalent to $T(r,f)$ up to a  bounded term, here $\omega_{FS}$ is the Fubini-Study form on $\mathbb P^1(\mathbb C).$ We need  the following proximity function 
 \begin{equation}\label{ttt1}
m_f(r, a)=\int_{|z|=r}\log\frac{1}{\|f,a\|}\frac{d\theta}{2\pi},
\end{equation}
where $\|\cdot , \cdot\|$ is the spherical distance on $\mathbb P^1(\mathbb C).$
 Taking the holomorphic form
$\mathscr S=dz,$ we will see that $\mathbb C$ is   $\mathscr S$-exhausted (see Definition \ref{ddi}). 
   Define
 $$\hat z=\int_0^z\mathscr S,$$
 which is a holomorphic function in $z$ with  a  unique zero $z=0.$ 
The  $\mathscr S$-disc and $\mathscr S$-circle of radius $r$ centered at 0 are  defined respectively by 
 $$D^{\mathscr S}(r)=\big\{z: |\hat z|<r\big\}, \ \ \  C^{\mathscr S}(r)=\big\{z: |\hat z|=r\big\}.$$
\ \ \ \  Let $g_r(0, z)$ be the Green function of $\Delta/2$ for $D^{\mathscr S}(r)$ with a pole at $0$ and Dirichlet boundary condition and  let $d\pi_0^r$ be the harmonic measure on $C^{\mathscr S}(r)$ with respect to 0. 
Notice that $g_r(0,z)=(1/\pi)\log(r/|\hat z|),$ $d\pi_0^r=d\theta/2\pi,$  
then by integration by part,   (\ref{ttt})  can be rewritten as
 \begin{eqnarray*}
T_f(r)&=&\int_1^r\frac{dt}{t}\int_{D^{\mathscr S}(t)}f^*\omega_{FS} \\
&=&\pi\int_{D^{\mathscr S}(r)} g_r(0,z)f^*\omega_{FS}-\pi\int_{D^{\mathscr S}(1)} g_1(0,z)f^*\omega_{FS},
 \end{eqnarray*}
and (\ref{ttt1}) can be replaced by 
$$m_f(r, a)=\int_{C^{\mathscr S}(r)}\log\frac{1}{\|f,a\|}d\pi_0^r.
$$
Similarly,  the counting function  is that 
$$N_f(r,a)=\int_1^r\frac{n^{\mathscr S}_f(t,a)}{t}dt,$$
where $n^{\mathscr S}_f(r,a)$ denotes the number of  zeros of $f-a$ on $D^{\mathscr S}(r)$.

Follow the  idea as shown as above,  we consider a non-compact Hermitian Riemann surface $(\mathcal S, h).$ Choosing $\mathscr S,$  a nowhere-vanishing holomorphic form 
  on $\mathcal S$ such that $\hat x=\int_o^x\mathscr S$ defines a holomorphic function, where $o$ is a fixed point in $\mathcal S.$ Following \cite{GR}, such  a form always exists. 
If $\mathcal S$ is \emph{$\mathscr S$-exhausted},  i.e., any sequence of $\mathscr S$-discs exhausts $\mathcal S$ when radius approaches increasingly to $R^{\mathscr S},$ where $R^{\mathscr S}$ is  the $\mathscr S$-radius of $\mathcal S$ with respect to $o,$ see Definition \ref{ppp}. 
To a  $\mathscr S$-exhausted surface $\mathcal S,$   $T_f^{\mathscr S}(r), m_{f}^{\mathscr S}(r, a)$ and $N_f^{\mathscr S}(r, a)$ of  a meromorphic function $f$  on $\mathcal S$ can be similarly defined and they make sense, see definition  for notations in Section 3.1.1. By computing Green functions and harmonic measures, we shall establish an analogue of Nevanlinna theory, which turns out  to be an extension  of  the classical Nevanlinna theory for $\mathbb C$ and $\mathbb D$ (unit disc), see, e.g., \cite{Noguchi, Nev, ru, ru-sibony}.
\subsection{Main results}~

In what follows, we state the main results of the paper.

\begin{thai}[Theorem \ref{thm1}]  Let $(\mathcal S, h; \mathscr S)$ be a  $\mathscr S$-exhausted Hermitian Riemann surface of $\mathscr S$-radius $R^{\mathscr S}$ with respect to $o.$  Let  $\gamma$ be an integrable function on $(0, R^{\mathscr S})$ with $\int_0^{R^{\mathscr S}}\gamma(r) dr=\infty.$
Let $f$ be a nonconstant meromorphic function on $\mathcal S$ and   $a_1, \cdots, a_q$ be distinct values in $\overline{\mathbb C}.$ Then for any $\delta>0$
 \begin{eqnarray*}
&& (q-2)T^{\mathscr S}_f(r)+T^{\mathscr S}(r, \mathscr R) \\
&\leq& \sum_{j=1}^q \overline{N}^{\mathscr S}_f(r,a_j)+O\Big(\log T^{\mathscr S}_f(r)+\log\|\mathscr S\|_{r,\sup}+\log\gamma(r)+\delta\log r\Big)
 \end{eqnarray*}
 holds for all $r\in(0,R^{\mathscr S})$ outside a set $E_\delta$ with 
$\int_{E_{\delta}}\gamma(r)dr<\infty,$ where  
$$\|\mathscr S\|_{r, \sup}=\sup\big\{\|\mathscr S_x\|_h: |\hat x|<r\big\}.$$
\end{thai}
In  Theorem I, the term $T^{\mathscr S}(r, \mathscr R)$ is called the \emph{characteristic} of the Ricci form $\mathscr R:=-dd^c\log h$ of $\mathcal S,$ which depends on the 
Gauss curvature of $\mathcal S,$ see (\ref{ric}).
 Notice that $\mathbb C$ is $dz$-exhausted  and $\|dz\|=1$ under  standard Euclidean metric. 
In this case, the Nevanlinna's functions agree with the classical ones.  Hence, Theorem I yields  (by letting 
 $\gamma=1$) a classical consequence for $\mathbb C$  that 
\begin{cori}  Let $f$ be a nonconstant meromorphic function on $\mathbb C,$  and let $a_1, \cdots, a_q$ be distinct points in $\overline{\mathbb C}.$  Then for any $\delta>0$
$$ (q-2)T_f(r) 
\leq \sum_{j=1}^q \overline{N}_f(r,a_j)+O\Big(\log T_f(r)+\delta\log r\Big)
$$ holds for all $r\in(0,\infty)$ outside a set $E_\delta$ of finite Lebesgue measure.  
\end{cori} 
Equipping $\mathbb D=\{z\in\mathbb C: |z|<1\}$ with standard Euclidean metric, then $\mathbb D$ is a  $dz$-exhausted Riemann surface. Take $\gamma=(1-r)^{-1},$  it concludes another classical consequence 
for $\mathbb D$ that
\begin{corii} Let $f$ be a nonconstant meromorphic function on $\mathbb D,$ and let $a_1, \cdots, a_q$ be distinct points in $\overline{\mathbb C}.$  Then for any $\delta>0$
$$ (q-2)T_f(r) 
\leq \sum_{j=1}^q \overline{N}_f(r,a_j)+O\Big(\log T_f(r)+\log \frac{1}{1-r}\Big)
$$ holds for all $r\in(0,1)$ outside a set $E_\delta$ with $\int_{E_\delta}(1-r)^{-1}dr<\infty$. 
\end{corii}

Associating a  holomorphic mapping $f: \mathcal S\rightarrow \mathcal R,$ where $(\mathcal R, \omega)$  is a compact Hermitian Riemann surface. 
Follow a theorem of Chern \cite{chern},  we can  similarly define  Nevanlinna's functions via $\mathscr S.$ 
We  prove a more generalized theorem of Theorem I as follows  
\begin{thaii}[Theorem \ref{thm2}] Let $(\mathcal S, h; \mathscr S)$ be a  $\mathscr S$-exhausted Hermitian Riemann surface of $\mathscr S$-radius $R^{\mathscr S}$ with respect to $o,$  and let  $\mathcal R$ be a compact  Riemann surface of genus $g.$
Fix a positive $(1,1)$-form  $\omega$  on $\mathcal R.$
Let  $\gamma$ be an integrable function on $(0, R^{\mathscr S})$ with $\int_0^{R^{\mathscr S}}\gamma(r) dr=\infty.$
Let $f$ be a nonconstant holomorphic mapping from  $\mathcal S$ into $\mathcal R$  and  $a_1, \cdots, a_q$ be distinct points in $\mathcal R.$ Then for any $\delta>0$
 \begin{eqnarray*}
&& (q-2+2g)T^{\mathscr S}_{f,\omega}(r)+T^{\mathscr S}(r, \mathscr R) \\
&\leq& \sum_{j=1}^q \overline{N}^{\mathscr S}_f(r,a_j)+O\Big(\log T^{\mathscr S}_{f,\omega}(r)+\log\|\mathscr S\|_{r,\sup}+\log\gamma(r)+\delta\log r\Big)
 \end{eqnarray*}
 holds for all $r\in(0,R^{\mathscr S})$ outside a set $E_\delta$ with 
$\int_{E_{\delta}}\gamma(r)dr<\infty,$ where  
$$\|\mathscr S\|_{r, \sup}=\sup\big\{\|\mathscr S_x\|_h: |\hat x|<r\big\}.$$
\end{thaii}

\begin{remark} 
In order to better estimate $\|\mathscr S\|,$ we can put Hermitian metric $\alpha=\frac{\sqrt{-1}}{2}\mathscr S\wedge \overline{\mathscr S}$ induced by $\mathscr S$ on $\mathcal S.$ Under this metric, we have $\|\mathscr S\|\equiv1.$
Then the result in Theorem II  becomes 
 \begin{eqnarray*}
&& (q-2+2g)T^{\mathscr S}_{f,\omega}(r)+T^{\mathscr S}(r, \mathscr R) \\
&\leq& \sum_{j=1}^q \overline{N}^{\mathscr S}_f(r,a_j)+O\Big(\log T^{\mathscr S}_{f,\omega}(r)+\log\gamma(r)+\delta\log r\Big). 
 \end{eqnarray*}

\end{remark}

We consider a defect relation of $f$ in Nevanlinna theory.  The \emph{simple defect} $\bar\delta_f(a)$ of $f$ with respect to $a$ is defined by 
\begin{equation}\label{defe}
\bar\delta_f(a)=1-\limsup_{r\rightarrow R^{\mathscr S}}\frac{\overline{N}_f^{\mathscr S}(r,a)}{T_{f,\omega}^{\mathscr S}(r)}.
\end{equation}
If $f$ is nonconstant, then we can check   $T_{f,\omega}^{\mathscr S}(r)\geq O(\log r).$ By using the First Main Theorem given in Section 3.1.2, we see that  $0\leq\bar\delta_f(a)\leq1.$

By estimating the lower bound of $T^{\mathscr S}(r, \mathscr R),$ we  obtain a defect relation
\begin{thaiii} Assume the same conditions as in Theorem II. Suppose, in addition,  that $-C\leq K\leq0$ for a non-negative constant $C.$
If $f$ satisfies 
$$\limsup_{r\rightarrow R^{\mathscr S}}\frac{Cr^2\|\mathscr S\|^{-1}_{r,\inf}+\log(\gamma(r)\|\mathscr S\|_{r,\sup})}{T_{f,\omega}^{\mathscr S}(r)}=0,$$
where $$\|\mathscr S\|_{r, \inf}=\inf\big\{\|\mathscr S_x\|_h: |\hat x|<r\big\},$$
then we have the defect relation 
$$\sum_{j=1}^q\bar\delta_f(a_j)\leq 2-2g.$$
\end{thaiii}

The earlier study of Nevanlinna theory for  non-compact Riemann surfaces  dates back to the work of  H. Wu \cite{wu} (also refer  to Shabat \cite{Shabat}). 
More  details  on this aspect,  the readers may refer  to the recent  very nice papers of He-Ru \cite{he-ru} and  P${\rm{\breve{a}}}$un-Sibony \cite{Sibony-Paun}.   
We have to indicate that the paper  presents  a  new method   of studying  Nevanlinna theory by putting a metric $h$ on $\mathcal S.$ However, 
 we  should   also point out  that    our results can be  derived by 
  H. Wu's method  (as well as He-Ru's arguments)     since 
  $\mathscr S$ produces an exhaustion function
$$\sigma(x):=\left|\int_o^x\mathscr S\right|$$ on $\mathcal S,$ so that  $\log \sigma^2$ is harmonic outside a compact set. 
 We  refer the readers to the standard arguments of He-Ru \cite{he-ru} without going into any details. 

\section{$\mathscr S$-exhausted Riemann sufaces}

\subsection{$\mathscr S$-exhausted Riemann sufaces}~

Let $\mathcal S$ be a non-compact  Riemann surface. Gunning proved  that 
\begin{proposition}[Gunning, \cite{GR}]\label{pp} Let $\mathcal S$ be a non-compact  Riemann surface. Then there is a holomorphic form $\mathscr S$ on $\mathcal S$ satisfying the following conditions

$(a)$ $\mathscr S$ has no zeros on $\mathcal S;$

$(b)$ $\int_\gamma\mathscr S=0$ for each smooth simple closed curve $\gamma$ in $\mathcal S.$
\end{proposition}
\begin{re} Condition  $(a)$ defines a holomorphic field $X$ without zeros on $\mathcal S$ in the following manner$:$ write 
$\mathscr S=\phi dz$ in a local holomorphic coordinate $z,$ it is trivial to  check that
$$X=\phi^{-1}\frac{\partial}{\partial z}$$
is well defined on $\mathcal S.$ Hence, $X$ is a  nowhere-vanishing holomorphic field that is dual to $\mathscr S.$
Condition  $(b)$ defines a holomorphic function $\hat x: \mathcal S\rightarrow \mathbb C$ by 
\begin{equation}\label{def}
L(x)=\int_o^x\mathscr S:=\hat x,
\end{equation}
where $o$ is a fixed point in $S.$ By $(a),$ we see that $o$ is a simple zero of $L.$
\end{re}
In this paper, besides   conditions $(a)$ and $(b),$ one  assumes that  $\mathscr S$ satisfies an additional condition: 
 $L(x)=0$ if and only if $x=o.$ This condition forces  $L$ to be a univalent function.

\begin{defi} Let $\mathscr S$ be a holomorphic form on $\mathcal S$ satisfying the conditions $(a)$ and $(b)$  in Proposition {\rm{\ref{pp}}}.  If, in addition, that $\mathscr S$ satisfies  the condition$:$ 
$L(x)=0$ if and only if $x=o,$
where $L$ is defined by $(\ref{def}),$ then we say that $\mathcal S$ is a $\mathscr S$-Riemann surface. 
\end{defi}
We define   $\mathscr S$-discs and  $\mathscr S$-circles on a $\mathscr S$-Riemann surface $\mathcal S.$ 
\begin{defi}  Let $(\mathcal S; \mathscr S)$ be a $\mathscr S$-Riemann surface. 
The $\mathscr S$-disc $D^{\mathscr S}(r)$ is defined by
$$D^{\mathscr S}(r)=\big\{x\in \mathcal S: |\hat x|<r\big\},$$
and the $\mathscr S$-circle $C^{\mathscr S}(r)$ is defined by
$$C^{\mathscr S}(r)=\big\{x\in \mathcal S: |\hat x|=r\big\}.$$
 \end{defi}
Since $L$ is univalent, then $D^{\mathscr S}(r)$ is simply connected for all $r>0.$ Notice that the case $C^{\mathscr S}(r)=\emptyset$ may happen if $r$ is sufficiently large, since $D^{\mathscr S}(r)$ could cover the whole surface  $\mathcal S$ when $r$ is large enough in some situations. For example, we consider the case where $\mathcal S=\mathbb D$ and $\mathscr S=dz,$ then  $C^{\mathscr S}(r)=\{z\in\mathbb D: |z|=r\}.$ If $r>1,$ then we have $C^{\mathscr S}(r)=\emptyset.$

\begin{defi}\label{ddi}  We say  that $\mathscr S$ is an exhaustion form, if $\{D^{\mathscr S}(r_n)\}_{n=1}^\infty$ exhausts $\mathcal S$ whenever $0<r_1<r_2<\cdots$ and $r_n\rightarrow +\infty$ as $n\rightarrow+\infty.$ Namely, 
$\mathcal S=\cup_{n=1}^\infty D^{\mathscr S}(r_n)$ and one of the following holds

$(i)$ \  $D^{\mathscr S}(r_1)\subset\overline{D^{\mathscr S}(r_1)}\subset D^{\mathscr S}(r_2)\subset\overline{D^{\mathscr S}(r_2)}\subset\cdots\subset\mathcal S;$

$(ii)$ \ there exists an integer  $k\geq1$ such that 
$$D^{\mathscr S}(r_1)\subset\overline{D^{\mathscr S}(r_1)}\subset\cdots\subset D^{\mathscr S}(r_k)=D^{\mathscr S}(r_{k+1})=\cdots=\mathcal S.$$
Precisely, we say that $\mathscr S$ is a parabolic exhaustion form if  condition $(i)$ is satisfied, and say that  $\mathscr S$ is a hyperbolic exhaustion form if  condition $(ii)$ is satisfied.
\end{defi}

\begin{defi}\label{ddi}  Let $(\mathcal S; \mathscr S)$ be a $\mathscr S$-Riemann surface.  We say that $\mathcal S$ is a $\mathscr S$-exhausted Riemann surface if $\mathscr S$ is an exhaustion form on $\mathcal S.$
\end{defi}

If $\mathcal S$ is a  $\mathscr S$-exhausted Riemann surface, then $C^{\mathscr S}(r)$  (for  $r>0$)  is a closed curve  contained properly  in $\mathcal S$ whenever 
$C^{\mathscr S}(r)\cap\mathcal S\not=\emptyset.$

\begin{defi}\label{ppp}  Let $(\mathcal S; \mathscr S)$ be a  $\mathscr S$-exhausted Riemann surface. Define $R^{\mathscr S}$ by 
$$R^{\mathscr S}=\sup\big\{r>0: C^{\mathscr S}(r)\cap\mathcal S\not=\emptyset\big\},$$
which is called the $\mathscr S$-radius of $\mathcal S$ with respect to $o.$ 
\end{defi}
From  the definition as above, we  have that $0<R^{\mathscr S}\leq\infty$ and  $C^{\mathscr S}(r)$ is a simple closed curve in $\mathcal S$ for $0<r<R^{\mathscr S}.$ 
Note  that the following is a direct consequence of Definition \ref{ddi} and Definition \ref{ppp}.
\begin{cor}\label{lem1}  Let $(\mathcal S; \mathscr S)$ be a  $\mathscr S$-exhausted Riemann surface of $\mathscr S$-radius $R^{\mathscr S}$ with respect to $o.$ Let $\{r_n\}_{n=1}^\infty$ be a   sequence of  positive integers $r_n$ with  
$r_1<r_2<\cdots$ and 
$r_n\rightarrow R^{\mathscr S}$ as $n\rightarrow +\infty.$ Then $\{D^{\mathscr S}(r_n)\}_{n=1}^\infty$ exhausts $\mathcal S.$
\end{cor}

\begin{ex} Let $\mathbb C$ be the standard complex Euclidean plane, then $\mathbb C$  is a $dz$-exhausted Riemann surface of $dz$-radius 
$\infty$ with respect to $0.$ We obtain $\hat z=z,$ $D^{dz}(r)=\{z\in\mathbb C: |z|<r\}$ and  $C^{dz}(r)=\{z\in\mathbb C: |z|=r\}.$ 
\end{ex}
\begin{ex} Equipping $\mathbb D$ with standard Euclidean metric, then $\mathbb D$  is  a  $dz$-exhausted Riemann surface of $dz$-radius $1$ with respect to $0.$ We have $\hat z=z,$ $D^{dz}(r)=\{z\in\mathbb D: |z|<r\}$ and  $C^{dz}(r)=\{z\in\mathbb D: |z|=r\}.$ 
\end{ex}

\subsection{Harmonic measures on $\mathscr S$-circles}~

Let $(\mathcal S, h; \mathscr S)$ be a  $\mathscr S$-exhausted Hermitian Riemann surface of $\mathscr S$-radius $R^{\mathscr S}$ with respect to $o.$  We denote by $\Delta$ the Laplace-Beltrami operator on $\mathcal S$ defined by $h.$
For $0<r<R^{\mathscr S},$
we shall compute   Green function $ g_r(o,x)$   of $\Delta/2$ for $D^{\mathscr S}(r)$ with Dirichlet boundary and a pole at $o,$ i.e., 
$$-\frac{1}{2}\Delta  g_r(o,x)=\delta_o(x), \ x\in D^{\mathscr S}(r); \ \  g_r(o,x)=0, \ x\in C^{\mathscr S}(r)$$
in the sense of distributions,  as well as  harmonic measure $d\pi^r_o$   on $C^{\mathscr S}(r)$ with respect to $o,$ for  $0<r<R^{\mathscr S}.$ We have the following well-known formula 
$$d\pi_o^r(x)=-\frac{1}{2}\frac{\partial  g_r(o,x)}{\partial \vec n}d\sigma_r(x),$$
where $\partial/\partial \vec{n}$ is the intward normal derivative on $C^{\mathscr S}(r).$  
\begin{lemma}\label{le1} For $0<r<R^{\mathscr S},$ we have
$$ g_r(o,x)=\frac{1}{\pi}\log\frac{r}{|\hat x|}.$$
\end{lemma}
\begin{proof}
Since  $L$ is a holomorphic function on $\mathcal S$ with a unique simple zero   $o,$ then  it  follows that  
$$\frac{1}{2\pi}\Delta\log|L(x)|=\delta_o(x), \  \ x\in  D^{\mathscr S}(r)$$
in the sense of distribution. 
Indeed,  it is clear that  $|L(x)|=r$ for $x\in C^{\mathscr S}(r).$ Hence, we obtain  
$$g_r(o,x)=\frac{1}{\pi}\log\frac{r}{|\hat x|}.$$
This completes the proof.
\end{proof}

\begin{lemma}\label{le2} For $0<r<R^{\mathscr S},$ we have
$$d\pi^r_o= \frac{\|\mathscr S\|_h}{2\pi r}d\sigma_r,$$
where $d\sigma_r$ is  the Riemannian line element on  $C^{\mathscr S}(r)$ defined by $h.$
\end{lemma}
\begin{proof}
Fix an arbitrary point $x_0\in C^{\mathscr S}(r),$ we   take a  holomorphic coordinate $z$  near $x_0$ such that $h|_{x_0}=1.$  Note that  
$$d\pi_o^r(x)=-\frac{1}{2}\frac{\partial  g_r(o,x)}{\partial \vec n}d\sigma_r(x),$$
where $\partial/\partial \vec{n}$ is the intward normal derivative on $\partial C^{\mathscr S}(r).$ By $|L(x_0)|=r$     
$$d\pi_o^r(x_0)=\frac{1}{2\pi r}\frac{\partial |L(x_0)|}{\partial \vec n}d\sigma_r(x_0).$$
Set $z=\xi_1+\sqrt{-1}\xi_2.$   Differentiating the equation $|L|=r$ near $x_0,$ then    
$$\frac{\partial|L|}{\partial \xi_1}d\xi_1+\frac{\partial|L|}{\partial \xi_2}d\xi_2=0,$$
which gives an inward normal vector 
$$\vec{n}|_{x_0}=\Bigg(\frac{\partial|L(x_0)|/\partial \xi_1}{\sqrt{\big (\frac{\partial|L(x_0)|}{\partial \xi_1}\big)^2+\big (\frac{\partial|L(x_0)|}{\partial \xi_2}\big)^2}}, \ 
\frac{\partial|L(x_0)|/\partial \xi_2}{\sqrt{\big (\frac{\partial|L(x_0)|}{\partial \xi_1}\big)^2+\big (\frac{\partial|L(x_0)|}{\partial \xi_2}\big)^2}}\Bigg).$$
Write $\mathscr S=\phi dz$ in the local coordinate $z$, a direct computation follows that 
$$\frac{\partial |L(x_0)|}{\partial \xi_1}=\Re\bigg[\phi(x_0)\frac{\overline{L(x_0)}^{1/2}}{L(x_0)^{1/2}}\bigg], \ \  
\frac{\partial |L(x_0)|}{\partial \xi_2}=-\Im\bigg[\phi(x_0)\frac{\overline{L(x_0)}^{1/2}}{L(x_0)^{1/2}}\bigg].
$$
Then 
$$\frac{\partial |L(x_0)|}{\partial \vec n}=\Big(\frac{\partial |L(x_0)|}{\partial \xi_1}, \frac{\partial |L(x_0)|}{\partial \xi_2}\Big)\cdot\vec n|_{x_0}=|\phi(x_0)|=\|\mathscr S{x_0}\|_h.$$
We conclude that 
$$d\pi^r_o(x_0)= \frac{\|\mathscr S{x_0}\|_h}{2\pi r}d\sigma_r(x_0).$$
This completes the proof.
\end{proof}

\section{Value distribution of meromorphic functions on $\mathcal S$}

\subsection{First Main Theorem}~ 

\subsubsection{Notations}~ 

Let $(\mathcal S, h; \mathscr S)$ be a  $\mathscr S$-exhausted Hermitian Riemann surface of $\mathscr S$-radius $R^{\mathscr S}$ with respect to $o.$ Using $\mathscr S,$ we  extend the notion of Nevanlinna's functions  to  $\mathcal S.$ 
Let $f$ be a meromorphic function on $\mathcal S$ and let  $r$ be satisfied with   $0<r_0<r<R^{\mathscr S}.$ Viewing $f=f_1/f_0=[f_0: f_1]$ as a holomorphic mapping from $\mathcal S$  into $\mathbb P^1(\mathbb C).$ We define the \emph{characteristic function}  of $f$ by
$$T^{\mathscr S}_f(r)=\int_{r_0}^r\frac{dt}{t}\int_{D^{\mathscr S}(t)}f^*\omega_{FS},$$
where $\omega_{FS}=dd^c\log(|\zeta_0|^2+|\zeta_1|^2)$ and 
$$d=\partial+\bar\partial, \ \ d^c=\frac{\sqrt{-1}}{4\pi}(\bar\partial-\partial), \ \ dd^c=\frac{\sqrt{-1}}{2\pi}\partial\bar\partial.$$
By integration by part,  it yields  that 
\begin{eqnarray*}
T^{\mathscr S}_f(r)&=&\int_{D^{\mathscr S}(r)} \log\frac{r}{|\hat x|}dd^c\log(|f_0(x)|^2+|f_1(x)|^2) \\
&&-\int_{D^{\mathscr S}(r_0)} \log\frac{r_0}{|\hat x|}dd^c\log(|f_0(x)|^2+|f_1(x)|^2).
\end{eqnarray*}
Locally, we write 
$$dV=\sqrt{-1}hdz\wedge d\bar z, \ \ \Delta=\frac{2}{h}\frac{\partial^2}{\partial z\partial\bar z}.$$
where $dV$ is the Riemannian area element  of $\mathcal S.$ 
We get  
$$dd^c\log(|f_0|^2+|f_1|^2)=\frac{1}{4\pi}\Delta\log(|f_0|^2+|f_1|^2)dV.$$
Notice (see Lemma \ref{le1}) that 
$$ g_r(o,x)=\frac{1}{\pi}\log\frac{r}{|\hat x|},$$
 then $T^{\mathscr S}_f(r)$ can be written in terms of Green function as follows
 \begin{eqnarray}\label{cha}
T^{\mathscr S}_f(r)&=&\frac{1}{4}\int_{D^{\mathscr S}(r)} g_r(o,x)\Delta\log(|f_0(x)|^2+|f_1(x)|^2)dV(x) \\
&&- \frac{1}{4}\int_{D^{\mathscr S}(r_0)} g_{r_0}(o,x)\Delta\log(|f_0(x)|^2+|f_1(x)|^2)dV(x). \nonumber
\end{eqnarray}
 Let  $a=[a_0:a_1]\in\mathbb P^1(\mathbb C)$ such that $f\not\equiv a.$  The \emph{proximity function} of $f$ with respect to $a$ is defined by
 \begin{equation}\label{pro}
m^{\mathscr S}_f(r,a)=\int_{C^{\mathscr S}(r)}\log\frac{1}{\|f,a\|}d\pi_o^r,
 \end{equation}
 where  $\|\cdot  , \cdot \|$ is the spherical distance on $\mathbb P^1(\mathbb C),$ defined by
 $$\|f, a\|=\frac{|\langle f; a \rangle|}{\|f\|\|a\|}=\frac{|a_0f_1-a_1f_0|}{\sqrt{|a_0|^2+|a_1|^2}\sqrt{|f_0|^2+|f_1|^2}},$$
 where 
 $$\langle f; a \rangle:=a_0f_1-a_1f_0.$$
Using Lemma \ref{le1}, we obtain 
$$m^{\mathscr S}_f(r,a)=\frac{1}{2\pi r}\int_{C^{\mathscr S}(r)}\|\mathscr S\|_h\log\frac{1}{\|f,a\|}d\sigma_r,$$
where $d\sigma_r$ is  the Riemannian line element on  $C^{\mathscr S}(r).$ 
The \emph{counting function} of $f$ with respect to $a$ is defined by 
$$N^{\mathscr S}_{f}(r,a)=\int_{r_0}^r\frac{n^{\mathscr S}_f(t, a)}{t}dt,$$
where $n_f^{\mathscr S}(r, a)$ denotes the number of the zeros of $f-a$ on $D^{\mathscr S}(r)$ counting multiplicities.  By Poincar\'e-Lelong formula \cite{gri}, we see that 
$$N^{\mathscr S}_{f}(r,a)=\int_{r_0}^r\frac{dt}{t}\int_{D^{\mathscr S}(t)}dd^c[\log|\langle f; a \rangle|^2],$$
where $dd^c[\log|\langle f; a \rangle|^2]$ is a current \cite{Noguchi}.  Similarly, in terms of Green function, there is an alternate expression 
 \begin{eqnarray}\label{cou}
N^{\mathscr S}_{f}(r,a)&=&\frac{1}{4}\int_{D^{\mathscr S}(r)} g_r(o,x)\Delta\log|\langle f(x); a \rangle|^2dV(x)  \\
&&-\frac{1}{4}\int_{D^{\mathscr S}(r_0)} g_{r_0}(o,x)\Delta\log|\langle f(x); a \rangle|^2dV(x),  \nonumber
 \end{eqnarray}
where $\Delta\log|\langle f; a \rangle|^2$ is  understood as a distribution. Similarly, one can define the \emph{simple counting function} 
$\overline{N}^{\mathscr S}_{f}(r,a)$ of $f$ with respect to $a,$ 
 which measures the size of the set of  zeros of $f-a$ without counting multiplicities. 

We define several  other symbols.  Let $\mathscr R$ be the Ricci form of $(\mathcal S, h),$ i.e.,
$$\mathscr R=Ric(\alpha)=-dd^c\log h,$$
where  
$$\alpha=\frac{\sqrt{-1}}{\pi}hdz\wedge d\bar z$$
is the K\"ahler form of  $\mathcal S.$ 
The \emph{characteristic} of  $\mathscr R$ is defined by 
$$T^{\mathscr S}(r, \mathscr R)=\int_{r_0}^r\frac{dt}{t}\int_{D^{\mathscr S}(t)}\mathscr R.$$
Since the Gauss curvature $K$  of $h$ is computed by 
$$K=-\frac{1}{2}\Delta\log h=-\frac{1}{h}\frac{\partial^2\log h}{\partial z\partial\bar z},$$
then we obtain  
$$2\mathscr R=K\alpha.
$$
By this with the above,  we see that the characteristic of  $\mathscr R$ can be expressed in terms of  Gauss curvature that 
 \begin{eqnarray}\label{ric}
T^{\mathscr S}(r, \mathscr R)&=&\frac{1}{2}\int_{D^{\mathscr S}(r)} g_r(o,x)K(x)dV(x) \\ 
&&  -\frac{1}{2}\int_{D^{\mathscr S}(r_0)} g_{r_0}(o,x)K(x)dV(x). \nonumber
 \end{eqnarray}
 
\subsubsection{First Main Theorem}~ 
  
 Let us introduce  Dynkin formula which is a generalization of 
 Green-Jensen formula \cite{noguchi, ru}, see  the probabilistic version of Dynkin formula in \cite{Dong, NN, itoo}.
 
\begin{lemma}[Dynkin formula] 
Let $u$ be a  function of  $\mathscr C^2$-class except  at most a polar set of singularities on a  Riemannian manifold $M.$ 
Let $D\subset M$ be a relatively compact domain  with  piecewise smooth boundary $\partial D.$ Assume that $u(o)\not=\infty$ for a fixed point $o\in D.$
Then 
$$\int_{\partial D}u(x)d\pi_o^{\partial D}(x)-u(o)=\frac{1}{2}\int_{D}g_D(o, x)\Delta u(x)dV(x),$$
where  $g_D(o, x)$ is the Green function of  Laplacian $\Delta/2$ for 
$D$ with a pole $o$ and Dirichlet boundary condition, and $d\pi_o^{\partial D}$ is the harmonic metric on $\partial D$ with respect to $o.$ Here, $\Delta u$ is understood as a distribution.
\end{lemma}

\begin{proof}
We are  to prove the above lemma by using a probabilistic approach, i.e.,  the probabilistic  Dynkin formula is applied to  showing this lemma. Let $X_t:=\{X_t\}_{t\geq0}$ be the Brownian motion  generated by 
 $\Delta/2$, started at $o\in M$ (see Section 2.2 in \cite{Dong}).   Denote by $\PP_o$ the law or distribution of $X_t$ starting from $o$ and by 
 $\mathbb E_o$ the expectation of $X_t$ with respect to $\mathbb P_o.$ Set the stopping time 
 $$\tau_D=\inf\big\{t>0: X_t\not\in D\big\}.$$
Note that the probabilistic  Dynkin formula  (see It\^o formula in \cite{Dong}, page 6) says that
$$\mathbb E_o\big[u(X_{\tau_D})\big]-u(o)=\frac{1}{2}\mathbb E_o\left[\int_0^{\tau_D}\Delta u(X_t)dt \right].$$
 On the other hand, using the co-area formula (see (2) in \cite{Dong}) and the relation between harmonic measures and hitting times (see (3) in \cite{Dong}),   we obtain   
 $$\mathbb E_o\left[\int_0^{\tau_D}\Delta u(X_t)dt \right]=\int_{D}g_D(o, x)\Delta u(x)dV(x)$$
 and 
 $$\mathbb E_o\big[u(X_{\tau_D})\big]=\int_{\partial D}u(x)d\pi_o^{\partial D}(x).$$
Substituting the two equalities into the  above probabilistic  Dynkin formula, then we have the lemma proved.
\end{proof}

Applying Dynkin formlua to $\log\|f, a\|^{-1}$ and 
noting  (\ref{cha})-(\ref{cou}),  it follows the First Main Theorem (F. M. T.) as follows
\begin{theorem}[F. M. T.]\label{FM} If $f\not\equiv a,$ then we have
$$T^{\mathscr S}_{f}(r)=m^{\mathscr S}_{f}(r,a)+N^{\mathscr S}_{f}(r,a)+O(1).$$
\end{theorem}

\begin{remark}  Theorem \ref{FM} can  be  confirmed by using  Green-Jensen formula instead of  Dynkin formula, since $\mathscr S$ can produce a parabolic or a hyperbolic exhaustion function $\sigma(x)=|\int_o^x\mathscr S|,$ see, e.g., \cite{he-ru, Sibony-Paun, Shabat, wu}. 

\end{remark}

\subsection{Second Main Theorem}~

The main purpose here is to prove the following S. M. T.:
\begin{theorem}\label{thm1} Let $(\mathcal S, h; \mathscr S)$ be a  $\mathscr S$-exhausted Hermitian Riemann surface of $\mathscr S$-radius $R^{\mathscr S}$ with respect to $o.$  Let  $\gamma$ be an integrable function on $(0, R^{\mathscr S})$ with $\int_0^{R^{\mathscr S}}\gamma(r) dr=\infty.$
Let $f$ be a nonconstant meromorphic function on $\mathcal S$ and  $a_1, \cdots, a_q$ be distinct values in $\overline{\mathbb C}.$ Then for any $\delta>0$
 \begin{eqnarray*}
&& (q-2)T^{\mathscr S}_f(r)+T^{\mathscr S}(r, \mathscr R) \\
&\leq& \sum_{j=1}^q \overline{N}^{\mathscr S}_f(r,a_j)+O\Big(\log T^{\mathscr S}_f(r)+\log\|\mathscr S\|_{r,\sup}+\log\gamma(r)+\delta\log r\Big)
 \end{eqnarray*}
 holds for all $r\in(0,R^{\mathscr S})$ outside a set $E_\delta$ with 
$\int_{E_{\delta}}\gamma(r)dr<\infty,$ where  
$$\|\mathscr S\|_{r, \sup}=\sup\big\{\|\mathscr S_x\|_h: |\hat x|<r\big\}.$$
\end{theorem}
 The following is called the  Borel's Growth Lemma. 
\begin{lemma}[\cite{ru-sibony}]\label{borel}  Let $\gamma$ be an integrable function on $(0, R)$ with $\int_0^R\gamma(r)dr=\infty.$  Let $h$ be a nondecreasing function of $\mathscr C^1$-class
on $(0, R).$ Assume that $\lim_{r\rightarrow R}h(r)=\infty$ and $h(r_0)>0$ for some $r_0\in (0, R).$ Then for any $\delta>0$  
$$h'(r)\leq h^{1+\delta}(r)\gamma(r)$$
holds for all $r\in(0,R)$ outside a set $E_\delta$ with $\int_{E_{\delta}}\gamma(r)dr<\infty.$ In particular, when $R=\infty,$ we can take $\gamma=1.$ Then  for any $\delta>0$  
$$h'(r)\leq h^{1+\delta}(r)$$
holds for all $r\in(0,\infty)$ outside a set $E_\delta$ of finite Lebesgue measure.
\end{lemma}

 We utilize Borel's Growth Lemma  to prove the following Calculus Lemma.
Let $k$ be a locally integrable function on $\mathcal S.$ Set
  $$E_k(r)=\int_{C^{\mathscr S}(r)}k d\pi_o^r , \ \ \ 
 A_{k}(r)=\int_{r_0}^{r}\frac{dt}{t}\int_{D^{\mathscr S}(t)}kdV.$$
\begin{lemma}\label{CL}  Let $(\mathcal S, h; \mathscr S)$ be a  $\mathscr S$-exhausted Hermitian Riemann surface of $\mathscr S$-radius $R^{\mathscr S}$ with respect to $o.$ Let  $\gamma$ be an integrable function on $(0, R^{\mathscr S})$ with $\int_0^{R^{\mathscr S}}\gamma(r) dr=\infty.$ Let $k$ be a locally integrable function on $\mathcal S.$ Then for any $\delta>0$
$$E_{k}(r)\leq \frac{\|\mathscr S\|_{r, \sup}r^{\delta}\gamma^{2+\delta}(r)}{2\pi}A^{(1+\delta)^2}_{k}(r)$$
holds for all $r\in(0,R^{\mathscr S})$ outside a set $E_\delta$ with 
$\int_{E_{\delta}}\gamma(r)dr<\infty,$ where  
$$\|\mathscr S\|_{r, \sup}=\sup\big\{\|\mathscr S_x\|_h: |\hat x|<r\big\}.$$
\end{lemma}
\begin{proof} Notice that 
$$\int_{D^{\mathscr S}(r)}kdV=\int_0^rdt\int_{C^{\mathscr S}(t)}kd\sigma_r,$$
then it  follows from Lemma \ref{le2} that   
$$\frac{d}{dr}\Big(r\frac{dA_k}{dr}\Big)=\int_{C^{\mathscr S}(r)}k d\sigma_r\geq \frac{2\pi r}{\|\mathscr S\|_{r, \sup}}E_k(r).$$
Using Lemma \ref{borel} twice (first to $rA'_k$ and then to $A_k$),  then we obtain  
$$E_{k}(r)\leq \frac{\|\mathscr S\|_{r, \sup}r^{\delta}\gamma^{2+\delta}(r)}{2\pi}A^{(1+\delta)^2}_{k}(r).$$
\end{proof}

We begin to prove Theorem \ref{thm1}:
\begin{proof}
  Consider a singular volume form (see \cite{gri,gri1})
   \begin{equation}\label{form}
\Phi=\frac{C\omega_{FS}}{\prod_{j=1}^q\|\zeta, a_j\|^2\log^2\|\zeta, a_j\|^{-2}}
 \end{equation}
on $\overline{\mathbb C},$  where     $a_1,\cdots,a_q$ are distinct values in  $\overline{\mathbb C}$ and 
 $\omega_{FS}=dd^c\log(1+|\zeta|^2).$
 Since $\overline{\mathbb C}$ is compact,  we can choose a positive number $C$  such that
$\int_{\overline{\mathbb C}}\Phi=1.
$
Set 
$$f^*\Phi=\xi\frac{\sqrt{-1}}{\pi}hdz\wedge d\bar z.
$$
By taking the Ricci form of  both sides of the above identity,  it follows from $Ric(\omega_{FS})=2\omega_{FS}$ that 
  \begin{eqnarray}\label{be}
dd^c[\log\xi] &=& (q-2)f^*\omega_{FS}-\sum_{j=1}^q(f-a_j)_0+D_{f, ram} \\
&& +\mathscr R -2\sum_{j=1}^qdd^c\log\log\|f,a_j\|^{-2} \nonumber
  \end{eqnarray}
 in the sense of currents, where $(f-a_j)_0$ is the zero divisor of $f-a_j,$ $D_{f, ram}$ is the ramification divisor of $f,$ and $\mathscr R=-dd^c\log h$ is the Ricci form 
of $\mathcal S.$ Applying the integral operator 
$$\int_{r_0}^r\frac{dt}{t}\int_{D^\mathscr S(t)}\cdot $$
to the above identity and using Dynkin formula, we get 
  \begin{eqnarray*}
\frac{1}{2}\int_{C^{\mathscr S}(r)}\log\xi d\pi^r_o&=& (q-2)T_f^{\mathscr S}(r)-\sum_{j=1}^qN^{\mathscr S}_f(r, a_j)+N^{\mathscr S}(r, D_{f, ram})
\\
&&+T^{\mathscr S}(r,\mathscr R) -\sum_{j=1}^q\int_{C^{\mathscr S}(r)}\log\log\|f,a_j\|^{-2}d\pi_o^r+O(1) \\
&\geq&  (q-2)T_f^{\mathscr S}(r)-\sum_{j=1}^q\overline{N}^{\mathscr S}_f(r, a_j)+
T^{\mathscr S}(r, \mathscr R) \\
&&-\sum_{j=1}^q\int_{C^{\mathscr S}(r)}\log\log\|f,a_j\|^{-2}d\pi_o^r+O(1).
  \end{eqnarray*}
The concavity of  $\log$  implies that 
  \begin{eqnarray}\label{hhh}
\int_{C^{\mathscr S}(r)}\log\log\|f,a_j\|^{-2}d\pi_o^r &\leq& \log\int_{C^{\mathscr S}(r)}\log\|f,a_j\|^{-2}d\pi_o^r \\
&=&  \log m^{\mathscr S}_f(r, a_j)+O(1)  \nonumber \\ 
&\leq& \log T^{\mathscr S}_f(r)+O(1). \nonumber
  \end{eqnarray}
  By this, we  obtain 
 \begin{eqnarray}\label{inq1}
\frac{1}{2}\int_{C^{\mathscr S}(r)}\log\xi d\pi^r_o &\geq&  (q-2)T^{\mathscr S}_f(r)-\sum_{j=1}^q\overline{N}^{\mathscr S}_f(r, a_j)+
T^{\mathscr S}(r, \mathscr R) \\
&& +O\big(\log T^{\mathscr S}_f(r)\big)+O(1).\nonumber
  \end{eqnarray}
\ \ \ \  It remains to estimate  the upper bound of the   term on the left hand side of (\ref{inq1}). By Lemma \ref{CL},  for any $\delta>0$
 \begin{eqnarray*}
\int_{C^{\mathscr S}(r)}\log\xi d\pi_o^r&\leq& \log\int_{C^{\mathscr S}(r)}\xi d\pi^r_o \\
&\leq& (1+\delta)^2\log\int_{r_0}^r\frac{dt}{t}\int_{D^{\mathscr S}(t)}\xi dV+\log\|\mathscr S\|_{r,\sup} \\
&&+(2+\delta)\log\gamma(r)+\delta\log r+O(1) \\
&=& (1+\delta)^2\log\int_{r_0}^r\frac{dt}{t}\int_{D^{\mathscr S}(t)}f^*\Phi+\log\|\mathscr S\|_{r,\sup} \\
&&+(2+\delta)\log\gamma(r)+\delta\log r+O(1)
  \end{eqnarray*}
holds for all $r\in(0,R^{\mathscr S})$ outside a set $E_\delta$ with 
$\int_{E_{\delta}}\gamma(r)dr<\infty,$ here $\gamma$ is an integrable function on $(0, R^{\mathscr S})$ such that $\int_0^{R^{\mathscr S}}\gamma(r) dr=\infty.$
Notice that 
$$\int_{D^{\mathscr S}(r)}f^*\Phi=\int_{\overline{\mathbb C}}n^{\mathscr S}_f(r, \zeta)\Phi(\zeta)$$
 due to the change of variable formula.  Using Fubini theorem, then 
 $$\int_{r_0}^r\frac{dt}{t}\int_{D^{\mathscr S}(t)}f^*\Phi=\int_{\overline{\mathbb C}}N^{\mathscr S}_f(r, \zeta)\Phi(\zeta)\leq c_0T^{\mathscr S}_f(r)$$
for some positive constant $c_0$.
Combining the above,  we conclude that 
  \begin{eqnarray}\label{inq2}
 \int_{C^{\mathscr S}(r)}\log\xi d\pi^r_o 
&\leq&  (1+\delta)^2\log T^{\mathscr S}_f(r)+\log\|\mathscr S\|_{r,\sup}  \\
&&+(2+\delta)\log\gamma(r)+\delta\log r+O(1) \nonumber
  \end{eqnarray}
  holds for all $r\in(0,R^{\mathscr S})$ outside a set $E_\delta$ with 
$\int_{E_{\delta}}\gamma(r)dr<\infty.$ 
Put together (\ref{inq1}) and (\ref{inq2}), then we prove the theorem.
\end{proof}

By uniformization theorem, for each non-negative constant $C,$ there  exists a metric  $h$ such that the Gauss curvature $K$ satisfies $-C\leq K\leq 0.$
Now, we estimate
 the lower bound of $T^{\mathscr S}(r, \mathscr R).$
 By Lemma \ref{le1} and Lemma \ref{le2}, it follows from  (\ref{ric}) that 
 \begin{eqnarray*}
T^{\mathscr S}(r, \mathscr R)&=&\frac{1}{2}\int_{D^{\mathscr S}(r)} g_r(o,x)K(x)dV(x)  -\frac{1}{2}\int_{D^{\mathscr S}(r_0)} g_{r_0}(o,x)K(x)dV(x) \\
&\geq& -\frac{C}{2}\int_{0}^rdt\int_{C^{\mathscr S}(t)}g_r(o,x)d\sigma_t(x) \\
&=& -C\int_{0}^rdt\int_{C^{\mathscr S}(t)}\frac{t}{\|\mathscr S\|_h}\log\frac{r}{t}d\pi_o^t  \\
&\geq& -\frac{C}{\|\mathscr S\|_{r,\inf}}\int_{0}^rt\log\frac{r}{t}dt \\
&=& -\frac{Cr^2}{4\|\mathscr S\|_{r,\inf}},
 \end{eqnarray*}
 where 
 $$\|\mathscr S\|_{r, \inf}=\inf\big\{\|\mathscr S_x\|_h: |\hat x|<r\big\}.
$$
Combining this with Theorem \ref{thm1}, we obtain 
\begin{theorem}\label{nnmm} Assume the same conditions as in Theorem $\ref{thm1}$. Suppose, in addition,  that $-C\leq K\leq0$ for a non-negative constant $C.$ Then for any $\delta>0$
 \begin{eqnarray*}
 (q-2)T^{\mathscr S}_f(r) &\leq& \sum_{j=1}^q \overline{N}^{\mathscr S}_f(r,a_j)  +O\Big(\log T^{\mathscr S}_f(r)
 +Cr^2\|\mathscr S\|^{-1}_{r,\inf} 
 \\
 && +\log\|\mathscr S\|_{r,\sup}
+\log\gamma(r)+\delta\log r\Big)
 \end{eqnarray*}
 holds for all $r\in(0,R^{\mathscr S})$ outside a set $E_\delta$ with 
$\int_{E_{\delta}}\gamma(r)dr<\infty,$ where  
 \begin{eqnarray*}
\|\mathscr S\|_{r, \inf}&=&\inf\big\{\|\mathscr S_x\|_h: |\hat x|<r\big\}, \\
\|\mathscr S\|_{r, \sup}&=&\sup\big\{\|\mathscr S_x\|_h: |\hat x|<r\big\}.
 \end{eqnarray*}
\end{theorem}

We recall  the definition of  the \emph{simple defect} $\bar\delta_f(a)$ in (\ref{defe}).
\begin{cor} Assume the same conditions as in Theorem $\ref{thm1}$. Suppose, in addition,  that $-C\leq K\leq0$ for a non-negative constant $C.$
If $f$ satisfies 
$$\limsup_{r\rightarrow R^{\mathscr S}}\frac{Cr^2\|\mathscr S\|^{-1}_{r,\inf}+\log(\gamma(r)\|\mathscr S\|_{r,\sup})}{T_{f}^{\mathscr S}(r)}=0,$$
then we have the defect relation 
$$\sum_{j=1}^q\bar\delta_f(a_j)\leq 2.$$
\end{cor}
\section{Targets are compact Riemann surfaces}

Let $\mathcal R$ be a compact  Riemann surface of genus $g.$ Fix  a positive (1,1)-form  $\omega$  on $\mathcal R.$ In a local holomorphic coordinate $\zeta,$ we may write  $\omega$ as
 $$\omega=\frac{\sqrt{-1}}{2\pi}wd\zeta\wedge d\bar\zeta.$$
 According to Chern's theorem \cite{chern}, we have that for every $a\in \mathcal R,$ there exists  a  positive function $u_a$  on $\mathcal R$ such that 
\begin{equation}\label{xx}
2dd^c[\log u_a]=\omega-\delta_a,
\end{equation}
 where $\delta_a$ should be understood as a current. Let $f: \mathcal S\rightarrow \mathcal R$ be a holomorphic mapping. 
 For  $a\in \mathcal R$ with $f\not\equiv a,$      by replacing $\|f, a\|^{-1}$ by $u_a\circ f,$ we  define similarly the Nevanlinna's functions of $f$ as follows
  \begin{eqnarray*}
 T^{\mathscr S}_{f, \omega}(r)&=&\int_{r_0}^r\frac{dt}{t}\int_{D^{\mathscr S}(t)}f^*\omega, \\
 m^{\mathscr S}_{f, \omega}(r,a)&=&\int_{C^{\mathscr S}(r)}\log (u_a\circ f)d\pi_o^r, \\ 
 N^{\mathscr S}_{f}(r,a)&=&\int_{r_0}^r\frac{n^{\mathscr S}_f(t, a)}{t}dt.
  \end{eqnarray*}
Since 
$$n^{\mathscr S}_f(r, a)=\int_{D^{\mathscr S}(r)}f^*\delta_a,$$
then we get 
 $$N^{\mathscr S}_{f}(r,a)=\int_{r_0}^r\frac{dt}{t}\int_{D^{\mathscr S}(t)}f^*\delta_a.$$
Using (\ref{xx}) and Dynkin formula, we obtain 
$${\rm{F. \ M. \ T.}}  \  \ \  \ T^{\mathscr S}_{f,\omega}(r)=m^{\mathscr S}_{f,\omega}(r,a)+N^{\mathscr S}_{f}(r,a)+O(1).$$

In what follows we derive the S. M. T..   A computation gives that 
\begin{equation}\label{ff}
Ric(\omega)=K'\omega,
\end{equation}
where $K'$ is the Gauss curvature of $\omega.$ If $g=0,$ then $\mathcal R$ can be regarded  as $\mathbb P^1(\mathbb C).$ 
Since  
$C_1\omega_{FS}\leq\omega\leq C_2\omega_{FS}$ for two suitable positive constants, then we can confirm the theorem    by using the conclusion  
 proved in Theorem \ref{thm1}.  In the following, we assume that $g\geq1.$
We need to modify  the  form (\ref{form}) as 
\begin{equation*}\label{zzzz}
\Phi=\frac{C\omega}{\prod_{j=1}^q u_{a_j}^{-2}\log^2 u_{a_j}^{2}},
\end{equation*}
where $C$ is  chosen so that $\Phi$ is normalized. Set
  $$f^*\Phi=\xi\frac{\sqrt{-1}}{\pi}hdz\wedge d\bar z.
$$
Since (\ref{xx}) and (\ref{ff}), then   in the sense of currents, (\ref{be}) becomes  
  \begin{eqnarray}\label{xyz}
dd^c[\log\xi] &=& qf^*\omega-f^*(K'\omega)-\sum_{j=1}^q(f-a_j)_0+D_{f, ram} \\
&&+\mathscr R-2\sum_{j=1}^qdd^c\log\log (u_{a_j}^{2}\circ f), \nonumber
  \end{eqnarray} 
It is  similar to (\ref{hhh}), we have
$$\int_{C^{\mathscr S}(r)}\log\log u_{a_j}^{2}d\pi_o^r  
\leq \log T^{\mathscr S}_{f,\omega}(r)+O(1). $$
 Integrating both sides of (\ref{xyz}) and using Dynkin formula,  we get 
  \begin{eqnarray*}
\frac{1}{2}\int_{C^{\mathscr S}(r)}\log\xi d\pi^r_o &\geq&  qT^{\mathscr S}_{f,\omega}(r)-\int_{r_0}^r\frac{dt}{t}\int_{D^{\mathscr S}(t)}f^*(K'\omega)
-\sum_{j=1}^q\overline{N}^{\mathscr S}_f(r, a_j)   \\
&& +T^{\mathscr S}(r, \mathscr R) +O\big(\log T^{\mathscr S}_{f,\omega}(r)\big)+O(1).
  \end{eqnarray*}
 For the second term on the right hand side of the  above identity, we use the change of variable formula and Gauss-Bonnet formula to yield that  
   \begin{eqnarray*}
 \int_{r_0}^r\frac{dt}{t}\int_{D^{\mathscr S}(r)}f^*(K'\omega)&=& \int_{r_0}^r\frac{dt}{t}\int_{\mathcal R}n_f^{\mathscr S}(t, \zeta)K'(\zeta)\omega(\zeta) \\
&=&  \int_{\mathcal R}\bigg[\int_{r_0}^r\frac{n_f^{\mathscr S}(t, \zeta)}{t}dt\bigg]K'(\zeta)\omega(\zeta) \\
&=&  \int_{\mathcal R}N_f^{\mathscr S}(r, \zeta)K'(\zeta)\omega(\zeta)  \\
&\geq&  \int_{\mathcal R}\big(T_{f,\omega}^{\mathscr S}(r)+O(1)\big)K'(\zeta)\omega(\zeta) \\
 &=& (2-2g)T_{f,\omega}^{\mathscr S}(r)+O(1).
  \end{eqnarray*}
Thus, 
 \begin{eqnarray}\label{a1}
\frac{1}{2}\int_{C^{\mathscr S}(r)}\log\xi d\pi^r_o &\geq&  (q-2+2g)T^{\mathscr S}_{f,\omega}(r)-\sum_{j=1}^q\overline{N}^{\mathscr S}_f(r, a_j)   \\
&& +T^{\mathscr S}(r, \mathscr R) +O\big(\log T^{\mathscr S}_{f,\omega}(r)\big)+O(1). \nonumber
  \end{eqnarray}
On the other hand, it follows by using the similar arguments as in  derivation of  (\ref{inq2}) that 
  \begin{eqnarray}\label{a2}
 \int_{C^{\mathscr S}(r)}\log\xi d\pi^r_o 
&\leq&  (1+\delta)^2\log T^{\mathscr S}_{f,\omega}(r)+\log\|\mathscr S\|_{r,\sup}  \\
&&+(2+\delta)\log\gamma(r)+\delta\log r+O(1)  \nonumber
  \end{eqnarray}
  holds for all $r\in(0,R^{\mathscr S})$ outside a set $E_\delta$ with 
$\int_{E_{\delta}}\gamma(r)dr<\infty.$ 
Put together  (\ref{a1}) and (\ref{a2}), then  we conclude the following S. M. T. 
\begin{theorem}\label{thm2} Let $(\mathcal S, h; \mathscr S)$ be a  $\mathscr S$-exhausted Hermitian Riemann surface of $\mathscr S$-radius $R^{\mathscr S}$ with respect to $o,$  and  $\mathcal R$ be a compact  Riemann surface of genus $g.$ 
Fix a positive $(1,1)$-form  $\omega$  on $\mathcal R.$
Let  $\gamma$ be an integrable function on $(0, R^{\mathscr S})$ with $\int_0^{R^{\mathscr S}}\gamma(r) dr=\infty.$
Let $f: \mathcal S\rightarrow\mathcal R$ be a nonconstant holomorphic mapping and  $a_1, \cdots, a_q$ be distinct points in $\mathcal R.$ Then for any $\delta>0$
 \begin{eqnarray*}
&& (q-2+2g)T^{\mathscr S}_{f,\omega}(r)+T^{\mathscr S}(r, \mathscr R) \\
&\leq& \sum_{j=1}^q \overline{N}^{\mathscr S}_f(r,a_j)+O\Big(\log T^{\mathscr S}_{f,\omega}(r)+\log\|\mathscr S\|_{r,\sup}+\log\gamma(r)+\delta\log r\Big)
 \end{eqnarray*}
 holds for all $r\in(0,R^{\mathscr S})$ outside a set $E_\delta$ with 
$\int_{E_{\delta}}\gamma(r)dr<\infty,$ where  
$$\|\mathscr S\|_{r, \sup}=\sup\big\{\|\mathscr S_x\|_h: |\hat x|<r\big\}.$$
\end{theorem}

Similarly as Theorem \ref{nnmm},  we have
\begin{theorem}\label{who} Assume the same conditions as in Theorem $\ref{thm2}$. Suppose, in addition,  that $-C\leq K\leq0$ for a non-negative constant $C.$ Then for any $\delta>0$
 \begin{eqnarray*}
 (q-2+2g)T^{\mathscr S}_{f,\omega}(r) &\leq& \sum_{j=1}^q \overline{N}^{\mathscr S}_f(r,a_j)  +O\Big(\log T^{\mathscr S}_{f,\omega}(r)+Cr^2\|\mathscr S\|^{-1}_{r,\inf}
 \\
 && +\log\|\mathscr S\|_{r,\sup}
+\log\gamma(r)+\delta\log r\Big)
 \end{eqnarray*}
 holds for all $r\in(0,R^{\mathscr S})$ outside a set $E_\delta$ with 
$\int_{E_{\delta}}\gamma(r)dr<\infty,$ where  
 \begin{eqnarray*}
\|\mathscr S\|_{r, \inf}&=&\inf\big\{\|\mathscr S_x\|_h: |\hat x|<r\big\}, \\
\|\mathscr S\|_{r, \sup}&=&\sup\big\{\|\mathscr S_x\|_h: |\hat x|<r\big\}.
 \end{eqnarray*}
\end{theorem}

Theorem \ref{who}  derives  a defect relation,  i.e., Theorem III in Introduction.

\noindent\textbf{Acknowledgement.} The authors are  grateful to   
the referee for his valuable comments on this paper. 
\vskip\baselineskip

\vskip\baselineskip

\label{lastpage-01}
\end{document}